\documentclass[12pt,reqno]{amsart}
\usepackage{amsmath,amssymb,amsfonts}
\usepackage{eucal}
\usepackage{graphicx,psfrag,color}
\usepackage{hyperref}
\usepackage{verbatim} 
\usepackage{cancel}
\usepackage{float}

\usepackage{lineno,xcolor}
\allowdisplaybreaks

\newtheorem{lem}{Lemma}[section]
\newtheorem{cor}[lem]{Corollary}
\newtheorem{prop}[lem]{Proposition}
\newtheorem{thm}[lem]{Theorem}
  \newtheorem{exama}[section]{Example}

  \newtheorem{conja}{Conjecture}[section]

\numberwithin{equation}{section}
\numberwithin{table}{section}
\renewcommand\phi{\varphi}            
\renewcommand\epsilon{\varepsilon}

\newcommand\Eta{{\mathrm H}}

\begin{document}
\title{Oriented Gain Graphs, Line Graphs and Eigenvalues}

\author{Nathan Reff}
\address{Department of Mathematics\\
The College at Brockport: State University of New York\\
Brockport, NY 14420, USA}
\email{nreff@brockport.edu}
\keywords{Orientation on gain graphs, gain graph, complex unit gain graph, line graph, oriented gain graph, voltage graph, line graph eigenvalues}
\subjclass[2010]{Primary  05C22, Secondary 05C50, 05C76, 05C25}

\begin{abstract}
A theory of orientation on gain graphs (voltage graphs) is developed to generalize the notion of orientation on graphs and signed graphs.  Using this orientation scheme, the line graph of a gain graph is  studied.  For a particular family of gain graphs with complex units, matrix properties are established.  As with graphs and signed graphs, there is a relationship between the incidence matrix of a complex unit gain graph and the adjacency matrix of the line graph.
\end{abstract}

\date{\today}
\maketitle
\section{Introduction}

The study of acyclic orientations of a graph enjoys a rich history of remarkable results, including Greene's bijection between acyclic orientations and regions of an associated hyperplane arrangement \cite{zbMATH03604927}, as well as Stanley's theorem on the number of acyclic orientations \cite{MR0317988}.  In \cite{MR1120422}, Zaslavsky develops a more general theory of orientation on {\it signed graphs} (graphs with edges labeled either $+1$ or $-1$), and shows that regions defined by a signed graphic hyperplane arrangement correspond with acyclic orientations of particular signed graphs.  Zaslavsky's approach is intimately connected with the theory of oriented matroids \cite{MR1744046}.

A {\it biased graph} is a graph with a list of distinguished cycles, such that if two cycles in the list are contained in a theta graph, then the third cycle of the theta graph must also be in the list \cite{MR1007712}.  Recently, Slilaty has developed an orientation scheme for certain biased graphs and their matroids  \cite{MR2701091}, answering a general orientation question posed by Zaslavsky \cite{MR2017726}. 

A \emph{gain graph} is a graph with the additional structure that each orientation of an edge is given a group element, called a \emph{gain}, which is the inverse of the group element assigned to the opposite orientation.  More recently, Slilaty also worked on real gain graphs and their orientations \cite{DANSGGHA}, generalizing graphic hyperplane arrangements.

In this paper, we define a notion of orientation for gain graphs over an arbitrary group.  The orientation provides two immediate applications.  First, a natural method for studying line graphs of gain graphs.  And second, a well defined incidence matrix.  In particular, gain graphs with complex unit gains (also called {\it complex unit gain graphs}) have particularly nice matrix and eigenvalue properties, which were initially investigated in \cite{MR2900705}.

The approach here fits the original orientation methods of Zaslavsky on signed graphs, and opens possibilities for future research.  A major question left to answer is what an acyclic orientation is under this setup.  This could lead to further connections with hyperplane arrangements and matroids.  Additional questions for future projects are also posed throughout the paper.

The reader may also be interested in other recent independent investigations of complex unit gain graphs and specializations that appear in the literature.  These include {\it weighted directed graphs} \cite{MR2859913, MR2928567, MR2929181, MR3045222, MR3191876}, which consider gains from the fourth roots of unity instead of the entire unit circle, and so called {\it Hermitian graphs} used to study universal state transfer \cite{MR3217403}.  A study of the characteristic polynomial for gain graphs has also been conducted in \cite{MR2890908}.

The paper is organized as follows.  In Section \ref{Sec2}, a background on the theory of gain graphs is provided.  In Section \ref{SECTIONOrientedGainGraphDEF}, we develop oriented gain graphs. The construction introduced borrows from Edmonds and Johnson's definition of a bidirected graph \cite{MR0267898} and Zaslavsky's more general oriented signed graphs \cite{MR1120422}.  Using this construction the line graph of an oriented gain graph is defined and studied in Section \ref{LineGraphsofAbelianGainGraphs}.  If the gain group is abelian, the line graph of an oriented gain graph is used to define the line graph of a gain graph.  This generalizes Zaslavsky's definition of the line graph of a signed graph \cite{MITTOSSG}.

Finally, in Section \ref{SECTIONMatricesofORIENTEDGAINGRAPHS} we discuss several matrices associated to these various graphs and line graphs.  For complex unit gain graphs, we generalize the classical relationship known for graphs and signed graphs between the incidence matrix of a graph and the adjacency matrix of the line graph.  This is used to study the adjacency eigenvalues of the line graph.

\section{Background}\label{Sec2}

The set of oriented edges, denoted by $\vec{E}(\Gamma)$, contains two copies of each edge with opposite directions.  An oriented edge from $v_i$ to $v_j$ is denoted by $e_{ij}$.  Formally, a \emph{gain graph} is a triple $\Phi=(\Gamma,\mathfrak{G},\phi)$ consisting of an \emph{underlying graph} $\Gamma=(V,E)$, the \emph{gain group} $\mathfrak{G}$ and a function $\phi :\vec{E}(\Gamma) \rightarrow \mathfrak{G}$ (called the \emph{gain function}), such that $\phi(e_{ij})=\phi(e_{ji})^{-1}$.  For brevity, we write $\Phi=(\Gamma,\phi)$ for a gain graph if the gain group is clear, and call $\Phi$ a $\mathfrak{G}$-gain graph.

The \emph{circle group} is $\mathbb{T}=\{ z\in\mathbb{C} : |z| =1\}\leq \mathbb{C}^{\times}$.  One particular family of gain graphs that we will be interested in are $\mathbb{T}$-gain graphs (or \emph{complex unit gain graphs}).  The center of a group $\mathfrak{G}$ is denoted by $Z(\mathfrak{G})$.  

We will always assume that $\Gamma$ is simple.  The set of vertices is $V:=\{v_1,v_2,\ldots,v_n\}$.  Edges in $E$ are denoted by $e_{ij}=v_i v_j$.   Even though this is the same notation for an oriented edge from $v_i$ to $v_j$ it will always be clear whether an edge or oriented edge is being used. We define $n:=|V|$ and $m:=|E|$. 

A \emph{switching function} is any function $\zeta:V\rightarrow \mathfrak{G}$.  Switching the $\mathfrak{G}$-gain graph $\Phi=(\Gamma,\phi)$ means replacing $\phi$ by $\phi^{\zeta}$, defined by: $\phi^{\zeta}(e_{ij})=\zeta(v_i)^{-1} \phi(e_{ij}) \zeta(v_j)$; producing the $\mathfrak{G}$-gain graph $\Phi^{\zeta}=(\Gamma,\phi^{\zeta})$.  We say $\Phi_1$ and $\Phi_2$ are \emph{switching equivalent}, written $\Phi_1 \sim \Phi_2$, when there exists a switching function $\zeta$, such that $\Phi_2=\Phi_1^{\zeta}$.  Switching equivalence forms an equivalence relation on gain functions for a fixed underlying graph.  An equivalence class under this equivalence relation is called a \emph{switching class} of $\phi$, and is denoted by $[\Phi]$.

The gain of a walk $W=v_1e_{12}v_2e_{23}v_3\cdots v_{k-1}e_{k-1,k}v_k$ is $\phi(W)=\phi(e_{12})\phi(e_{23})\cdots\phi(e_{k-1,k})$.  
Switching conjugates the gain of a closed walk in a gain graph.
\begin{prop}\label{conjugateWalk} Let $W=v_1e_{12}v_2e_{23}v_3\cdots v_{k}e_{k1}v_1$ be a closed walk in a $\mathfrak{G}$-gain graph $\Phi$.  Let $\zeta:V\rightarrow \mathfrak{G}$.  Then $\phi^{\zeta}(W)=\zeta(v_1)^{-1} \phi(W)\zeta(v_1)$.
\end{prop}
\begin{proof} The following calculation verifies the result:
\begin{align*}
\phi^{\zeta}(W) &= \phi^{\zeta}(e_{12})\phi^{\zeta}(e_{23})\cdots\phi^{\zeta}(e_{k1})\\
&=\zeta(v_1)^{-1}\phi(e_{12})\zeta(v_2)\zeta(v_2)^{-1}\phi(e_{23})\cdots\zeta(v_k)^{-1}\phi(e_{k1})\zeta(v_1)\\
&= \zeta(v_1)^{-1}\phi(W)\zeta(v_1).\qedhere
\end{align*}
\end{proof}

If $C$ is a cycle in the underlying graph $\Gamma$, then the gain of $C$ in some fixed direction starting from a base vertex $v$ is denoted by $\phi(\vec{C}_v)$.  Thus, $\vec{C}_v$ is called a \emph{directed cycle with base vertex $v$}.  The \emph{fundamental cycle} associated with edge $e$, denoted $C_T(e)$, is the cycle obtained by adding $e$ to a maximal forest of $\Gamma$.  For abelian $\mathfrak{G}$ we can state a sufficient condition to guarantee two $\mathfrak{G}$-gain graphs are switching equivalent.  The proof here is an adaptation of a well known construction on signed graphs, see for example \cite[Theorem II.A.4]{Math581Notes}.



\begin{lem}\label{SwitchingFExistence} Let $\mathfrak{G}$ be abelian.  Let $\Phi_1$ and $\Phi_2$ be $\mathfrak{G}$-gain graphs with the same underlying graph $\Gamma$.  If for every cycle $C$ in $\Gamma$ there exists a directed cycle with base vertex $v$ such that $\phi_1(\vec{C}_v)=\phi_2(\vec{C}_v)$, then there exists a switching function $\zeta$ such that $\Phi_2=\Phi_1^{\zeta}$.
\end{lem}
\begin{proof} 
Switching individual components is independent, so we may assume that $\Gamma$ is connected.

Pick a spanning tree $T$ of $\Gamma$, and label the vertices so that for all $i\in \{2,\ldots, n\}$, $v_i$ is always adjacent to a vertex in $\{v_1,\ldots,v_{i-1}\}$.  

For all $i\in\{2,\ldots,n\}$, let $e_{ij}$ be the unique oriented edge in $\vec{E}(T)$ from $v_i$ to $v_j \in \Phi{:}\{v_1,\ldots,v_{i-1}\}$.  We define $\zeta:V\rightarrow \mathfrak{G}$ recursively by
\begin{equation*}
\zeta(v_i)=
\begin{cases} 1_{\mathfrak{G}} & \text{if }i=1,
\\
\phi_1(e_{ij})\zeta(v_j)\phi_2(e_{ij})^{-1} &\text{otherwise.}
\end{cases}
\end{equation*}
Therefore, for all $e_{ij}\in \vec{E}(T)$, $\phi_2(e_{ij})=\phi_1^{\zeta}(e_{ij})$.  Now we need to verify that $\phi_2$ and $\phi_1^{\zeta}$ agree on the oriented edges of $\Gamma$ that are not part of $\vec{E}(T)$.

Let $f\in E(C)\backslash E(T)$.  By hypothesis, there is some directed cycle with base vertex $v$ such that, $\phi_1(\vec{C}_v)=\phi_2(\vec{C}_v)$, and therefore, $\phi_1(\overrightarrow{C_T(f)}_v)=\phi_2(\overrightarrow{C_T(f)}_v)$.  Furthermore, by Proposition \ref{conjugateWalk}, $\phi_1^{\zeta}(\overrightarrow{C_T(f)}_v)=\phi_1(\overrightarrow{C_T(f)}_v)$ since $\mathfrak{G}$ is abelian.  Hence, $\phi_2(\overrightarrow{C_T(f)}_v)=\phi_1^{\zeta}(\overrightarrow{C_T(f)}_v)$.

Without loss of generality, suppose $\overrightarrow{C_T(f)}_v=e_{12}e_{23}\cdots e_{k-1,k}e_{k1}$ and $f=e_{k1}$.  Since $\phi_2(\overrightarrow{C_T(f)}_v)=\phi_1^{\zeta}(\overrightarrow{C_T(f)}_v)$ we can write
\[ \phi_2(e_{12})\phi_2(e_{23})\cdots \phi_2(e_{k-1,k})\phi_2(e_{k1})= \phi_1^{\zeta}(e_{12})\phi_1^{\zeta}(e_{23})\cdots \phi_1^{\zeta}(e_{k-1,k})\phi_1^{\zeta}(e_{k1}).\]
Also, since $\phi_2(e_{ij})=\phi_1^{\zeta}(e_{ij})$ for every $e_{ij}\in \vec{E}(T)$, by taking inverses we can simplify the product equality to $\phi_2(e_{k1})=
\phi_1^{\zeta}(e_{k1})$.  Hence, $\phi_2$ and $\phi_1^{\zeta}$ agree on all oriented edges.  
\end{proof} 

\section{Oriented Gain Graphs}\label{SECTIONOrientedGainGraphDEF}
A \emph{$\mathfrak{G}$-phased graph} is a pair $(\Gamma,\omega)$, consisting of a graph $\Gamma$, and a \emph{$\mathfrak{G}$-incidence phase function}
$\omega : V\times E\rightarrow \mathfrak{G}\cup \{0\}$ that satisfies
\begin{align*}
\omega(v,e)&\neq 0 \text{ if $v$ is incident to $e$},\\
\omega(v,e)&= 0 \text{ otherwise}.
\end{align*}

A $\mathfrak{G}$-phased graph with $\text{Image}(\omega)\subseteq \{+1,0,-1\}$ is called a \emph{bidirected graph}, and we call $\omega$ a \emph{bidirection}.  If $\beta$ is a bidirection then $\beta(v,e)=+1$ can be thought of as an edge-end arrow proceeding into $v$ and $\beta(v,e)=-1$ as an edge-end arrow exiting $v$.  See Figure \ref{GraphBidandPhasedEx} for some examples.

\begin{figure}[h!]
    \includegraphics[scale=0.65]{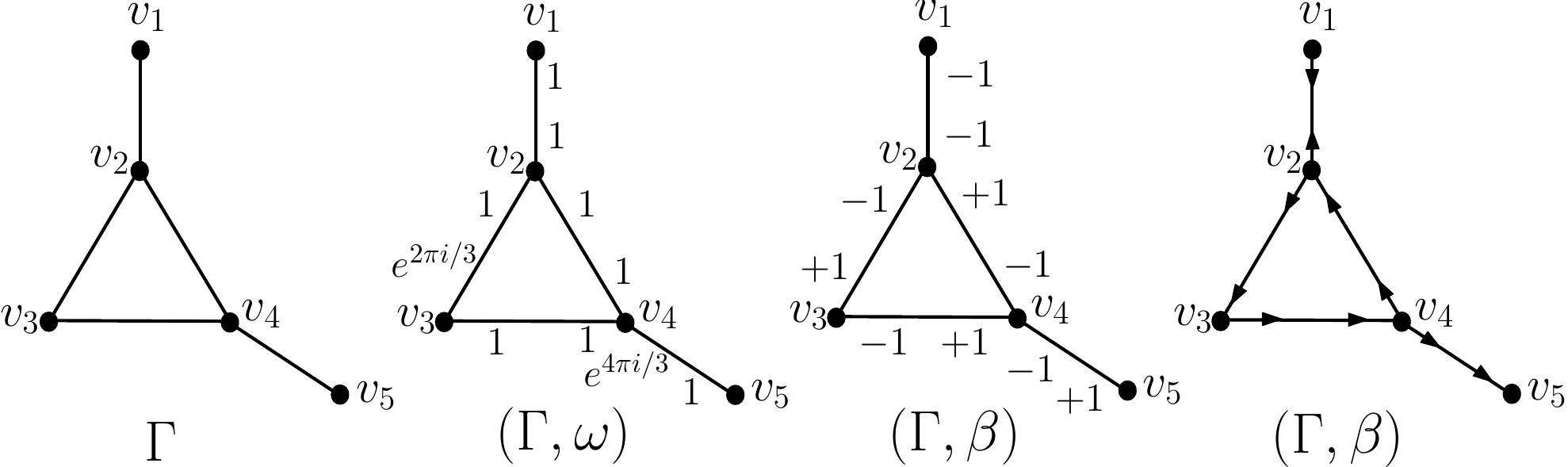}\centering
    \caption{A graph $\Gamma$, a $\mathbb{T}$-phased graph $(\Gamma,\omega)$, a bidirected graph $(\Gamma,\beta)$ and the same bidirected graph using the arrow convention instead of labeling the incidences.}\label{GraphBidandPhasedEx}
\end{figure}

A (weak) \emph{involution} of a group $\mathfrak{G}$ is any element $\mathfrak{s}\in\mathfrak{G}$ such that $\mathfrak{s}^2=1_{\mathfrak{G}}$.  Henceforth, all involutions are weak involutions, thus the group identity will also be called an involution.

Let $\mathfrak{G}^{\mathfrak{s}}$ be a group with a distinguished central involution $\mathfrak{s}$.  Formally, $\mathfrak{s}$ is a distinguished involution of $\mathfrak{G}^{\mathfrak{s}}$ such that $\mathfrak{s}\in Z(\mathfrak{G}^{\mathfrak{s}})$.  The notation $\mathfrak{G}^{\mathfrak{s}}$ is to clearly indicate which choice one has made for $\mathfrak{s}$, so the groups $\mathfrak{G}^{\mathfrak{s}}$ and $\mathfrak{G}^{\mathfrak{t}}$ may be equal even though $\mathfrak{s}\neq \mathfrak{t}$.  For example, $\mathbb{T}^1=\mathbb{T}^{-1}=\mathbb{T}$ as groups, but the distinguished central involution of interest changes when writing $\mathbb{T}^1$ and $\mathbb{T}^{-1}$. 

Let $\omega$ be a $\mathfrak{G}^{\mathfrak{s}}$-incidence phase function.  The \emph{$\mathfrak{G}^{\mathfrak{s}}$-gain graph associated to a $\mathfrak{G}^{\mathfrak{s}}$-phased graph $(\Gamma,\omega)$}, denoted by $\Phi(\omega)$, has its gains defined by
\begin{equation}\label{orientationREL}
\phi(e_{ij})=\omega(v_i,e_{ij})\cdot\mathfrak{s}\cdot\omega(v_j,e_{ij})^{-1}. 
\end{equation}

Notice that because $\mathfrak{s}$ is an involution,
\begin{align*}
\phi(e_{ij}) &= \omega(v_i,e_{ij})\cdot\mathfrak{s}\cdot\omega(v_j,e_{ij})^{-1} \\
&= \big[\omega(v_j,e_{ij})\cdot\mathfrak{s}^{-1}\cdot\omega(v_i,e_{ij})^{-1}\big]^{-1}\\ 
&= \big[\omega(v_j,e_{ij})\cdot\mathfrak{s}\cdot\omega(v_i,e_{ij})^{-1}\big]^{-1}\\
&=\phi(e_{ji})^{-1}.
\end{align*}
This means $\Phi(\omega)$ really is a $\mathfrak{G}^{\mathfrak{s}}$-gain graph with appropriately labeled oriented edges.  

See Figures \ref{GraphPhasedAssociated} and \ref{GraphPhasedAssociatedMINUS} for distinguishing examples of $\Phi(\omega)$.  To make the pictures less cluttered each edge will have only one oriented edge labelled with a gain (since the gain in the opposite direction is immediately determined as the inverse).

\begin{figure}[h!]
    \includegraphics[scale=0.65]{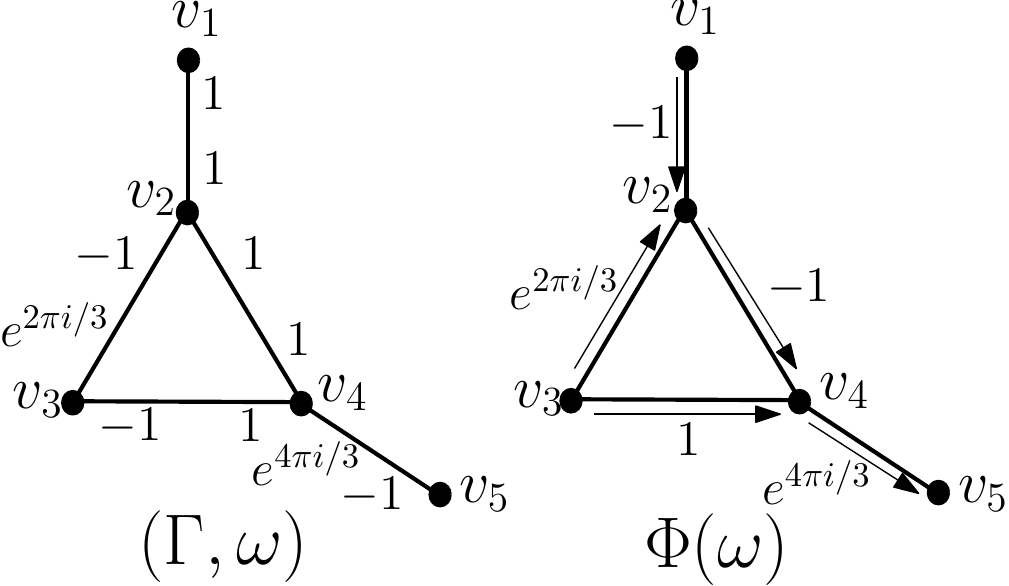}\centering
    \caption{A $\mathbb{T}^{-1}$-phased graph $(\Gamma,\omega)$ and the associated $\mathbb{T}^{-1}$-gain graph $\Phi(\omega)$.}\label{GraphPhasedAssociated}
\end{figure}

\begin{figure}[h!]
    \includegraphics[scale=0.65]{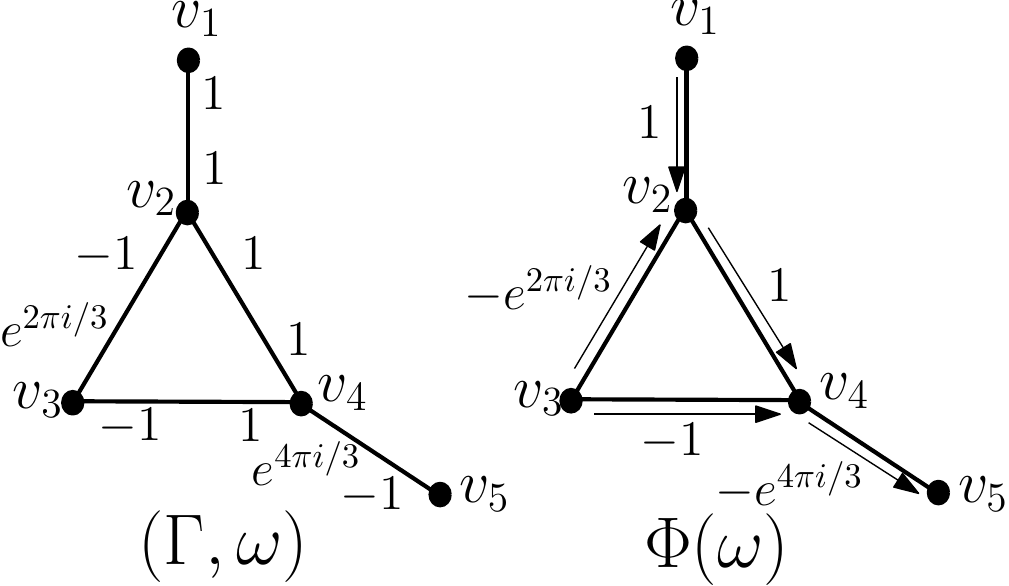}\centering
    \caption{A $\mathbb{T}^{1}$-phased graph $(\Gamma,\omega)$ and the associated $\mathbb{T}^{1}$-gain graph $\Phi(\omega)$.}\label{GraphPhasedAssociatedMINUS}
\end{figure}



For a $\mathfrak{G}^{\mathfrak{s}}$-gain graph $\Phi$, an \emph{orientation of $\Phi$} is any $\mathfrak{G}^{\mathfrak{s}}$-incidence phase function $\omega$ that satisfies Equation \eqref{orientationREL}.  An \emph{oriented $\mathfrak{G}^{\mathfrak{s}}$-gain graph} is a pair $(\Phi,\omega)$, consisting of a $\mathfrak{G}^{\mathfrak{s}}$-gain graph $\Phi$, and $\omega$, an orientation of $\Phi$.  To obtain an oriented signed graph \cite{MR1120422} one can choose $\mathfrak{G}=\{+1,-1\}$ with $\mathfrak{s}=-1$.

Switching a $\mathfrak{G}^{\mathfrak{s}}$-incidence phase function $\omega$ by $\zeta:V\rightarrow \mathfrak{G}^{\mathfrak{s}}$ means replacing $\omega$ by $\omega^{\zeta}$ defined by $\omega^{\zeta}(v_i,e_{ij})=\zeta(v_i)^{-1}\omega(v_i,e_{ij})$.  The $\mathfrak{G}^{\mathfrak{s}}$-gain graph associated to the switched $\omega^{\zeta}$ is the same as the $\zeta$-switched $\mathfrak{G}^{\mathfrak{s}}$-gain graph associated to $\omega$.  This generalizes the same result known for signed graphs \cite{MR1120422}.

\begin{prop} If $\zeta:V\rightarrow \mathfrak{G}^{\mathfrak{s}}$, then $\Phi(\omega^{\zeta})=\Phi(\omega)^{\zeta}$.
\end{prop}
\begin{proof} We compute $\phi(e_{ij})$ in $\Phi(\omega^{\zeta})$ as follows:
\begin{align*}
\phi_{\omega^{\zeta}}(e_{ij})&=\omega^{\zeta}(v_i,e_{ij})\cdot\mathfrak{s}\cdot\omega^{\zeta}(v_j,e_{ij})^{-1}\\
&=\zeta(v_i)^{-1}\omega(v_i,e_{ij})\cdot\mathfrak{s}\cdot\omega(v_j,e_{ij})^{-1}\zeta(v_j)\\
&=\zeta(v_i)^{-1}\phi_{\omega}(e_{ij})\zeta(v_j)\\
&=\phi_{\omega}^{\zeta}(e_{ij}).\qedhere
\end{align*}
\end{proof}

\noindent {\bf Question 1:} What is an acyclic orientation in an oriented $\mathfrak{G}^{\mathfrak{s}}$-gain graph?  The answer to this could have some interesting connections to matroids and hyperplane arrangement theory.  This study could further generalize the results of Greene \cite{zbMATH03604927}, Zaslavsky \cite{MR1120422} and Slilaty \cite{DANSGGHA}.

\section{The Line Graph of an Abelian Gain Graph}\label{LineGraphsofAbelianGainGraphs}

Let $\Lambda_{\Gamma}$ denote the line graph of the unsigned graph $\Gamma$.  Let $\omega_{\Lambda}$ be a $\mathfrak{G}^{\mathfrak{s}}$-incidence phase function on $\Lambda_{\Gamma}$ defined by 
\begin{equation}\label{LGincidencerel}
\omega_{\Lambda}(e_{ij},e_{ij}e_{jk})=\omega(v_j,e_{ij})^{-1}.
\end{equation}  
Let $(\Phi,\omega)$ be an oriented $\mathfrak{G}^{\mathfrak{s}}$-gain graph.  The \emph{line graph of $(\Phi,\omega)$} is the oriented $\mathfrak{G}^{\mathfrak{s}}$-gain graph $(\Phi(\omega_{\Lambda}),\omega_{\Lambda})$.  See Figure \ref{OLinegraphExample} for an example.
\begin{figure}[h!]
    \includegraphics[scale=0.6]{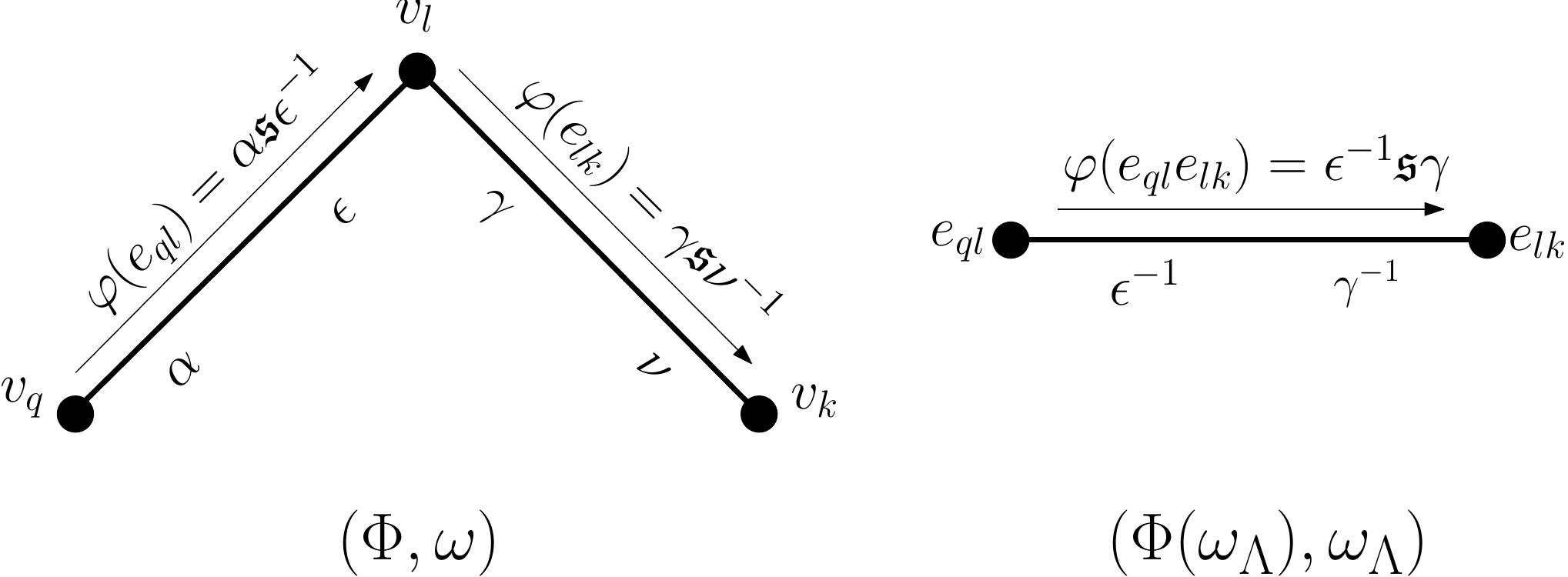}\centering
    \caption{An oriented $\mathfrak{G}^{\mathfrak{s}}$-gain graph $(\Phi,\omega)$ and its line graph $(\Phi(\omega_{\Lambda}),\omega_{\Lambda})$.}\label{OLinegraphExample}
\end{figure}


The line graph of an oriented gain graph $(\Phi,\omega)$ can be viewed as a generalization of the line graph a graph in the following way.  Suppose $\omega=\mathfrak{s}$ for every incidence of $\Phi$, thus making the gain of every edge $\mathfrak{s}$.  By Equation \eqref{LGincidencerel}, $\omega_{\Lambda}=\mathfrak{s}^{-1}=\mathfrak{s}$ for every incidence of $\Phi(\omega_{\Lambda})$.  Hence, the gain of every edge in the line graph is also $\mathfrak{s}$.

\begin{prop} The line graph of $(\Phi(\mathfrak{s}),\mathfrak{s})$ is $(\Phi(\mathfrak{s}_{\Lambda}),\mathfrak{s}_{\Lambda})$.
\end{prop}

A quick calculation can verify that changing orientations of a particular edge in an oriented gain graph corresponds to switching the associated vertex in the line graph.  The following more general result says that switching equivalent oriented gain graphs produce switching equivalent line graphs.  This generalizes the same result known for signed graphs \cite[Lemma 6.1]{MananthavadyNotes} and borrows from its proof methods.  This is an essential ingredient to defining the line graph of a gain graph $\Phi$ in general.   
\begin{thm}\label{ReOrSwEqLG} Let $\mathfrak{G}^{\mathfrak{s}}$ be abelian.  Let $\Phi_1$ and $\Phi_2$ be $\mathfrak{G}^{\mathfrak{s}}$-gain graphs with the same underlying graph $\Gamma$.  If $(\Phi_1,\omega)$ and $(\Phi_2,\kappa)$ are oriented $\mathfrak{G}^{\mathfrak{s}}$-gain graphs where $\Phi_1\sim\Phi_2$, then $\Phi(\omega_{\Lambda})\sim\Phi(\kappa_{\Lambda})$.
\end{thm}
\begin{proof}  

Let $(\Phi(\omega_{\Lambda}),\omega_{\Lambda})$ and $(\Phi(\kappa_{\Lambda}),\kappa_{\Lambda})$ be the line graphs of $(\Phi_1,\omega)$ and $(\Phi_2,\kappa)$, respectively.  Both line graphs have the same underlying graph $\Lambda_{\Gamma}$.

Let $C$ be a cycle in $\Lambda_{\Gamma}$.  Assume that $C=e_0e_1\cdots e_{l-1}e_0$, where edges $e_{i-1}$ and $e_{i}$ have a common vertex $v_i$ in $\Gamma$ for every $i\in\{1,\ldots,l-1\}$, and edges $e_{l-1}$ and $e_0$ have a common vertex $v_l$ in $\Gamma$. 

Now we calculate the gain of $\vec{C}_{e_0}$ in $(\Phi(\omega_{\Lambda}),\omega_{\Lambda})$ as follows:
\begin{align*}
\phi_{\omega_{\Lambda}}(\vec{C}_{e_0})&=\phi_{\omega_{\Lambda}}(e_0e_1)\phi_{\omega_{\Lambda}}(e_1e_2)\cdots\phi_{\omega_{\Lambda}}(e_{l-1}e_0)\\
&=(\omega_{\Lambda}(e_0,e_0e_1)\cdot\mathfrak{s}\cdot\omega_{\Lambda}(e_1,e_0e_1)^{-1})(\omega_{\Lambda}(e_1,e_1e_2)\cdot\mathfrak{s}\cdot\omega_{\Lambda}(e_2,e_1e_2)^{-1})\cdots\\
&\qquad\qquad(\omega_{\Lambda}(e_{l-1},e_{l-1}e_0)\cdot\mathfrak{s}\cdot\omega_{\Lambda}(e_0,e_{l-1}e_0)^{-1})\\
&=\mathfrak{s}^{l}\omega(v_1,e_0)^{-1}\omega(v_1,e_1)\omega(v_2,e_1)^{-1}\omega(v_2,e_2)\cdots\omega(v_l,e_{l-1})^{-1}\omega(v_l,e_0).
\end{align*}
Notice that the resulting product only involves the orientation $\omega$, and it appears that after rearranging the factors this product may simplify to the gain of a closed walk in $\Gamma$.  However, since $v_i$ may be the same as $v_{i+1}$, the product might not be the gain of a closed walk in $\Gamma$.  If $v_i=v_{i+1}$ for some $i$, then $\omega(v_i,e_i)\omega(v_{i+1},e_i)^{-1}=1_{\mathfrak{G}^{\mathfrak{s}}}$, so we can reduce the product above.  Notice that if $v_i=v_{i+1}$, then $v_i$ is incident to $e_{i-1}$, $e_{i}$ and $e_{i+1}$ in $\Gamma$; thus the vertices $e_{i-1}$, $e_{i}$ and $e_{i+1}$ in the line graph $\Lambda_{\Gamma}$ form a triangle.  Therefore, $C'=e_0e_1\cdots e_{i-1}e_{i+1}\ldots e_{l-1}e_0$ is a cycle in $(\Phi(\omega_{\Lambda}),\omega_{\Lambda})$ and $\phi_{\omega_{\Lambda}}(\vec{C'}_{e_0})=\mathfrak{s}\cdot\phi_{\omega_{\Lambda}}(\vec{C}_{e_0})$.

Let $C''=f_0 f_1 \cdots f_{k-1}f_0$ be the cycle in $(\Phi(\omega_{\Lambda}),\omega_{\Lambda})$ obtained from $C$ by reducing all consecutive equal vertices.  Suppose that for all $i\in\{1,\ldots,k-1\}$ the edges $f_{i-1}$ and $f_{i}$ have a common vertex $w_i$ in $\Gamma$, and edges $f_{k-1}$ and $f_0$ have common vertex $w_k$ in $\Gamma$.  Now notice that $W_{\Gamma}=w_1 f_1 w_2 \cdots f_{k-1} w_k f_0 w_1$ is a closed walk of length $k$ in $\Gamma$.  Therefore,
\begin{align*}
\phi_{\omega_{\Lambda}}(\vec{C''}_{f_0})&=\phi_{\omega_{\Lambda}}(f_0f_1)\phi_{\omega_{\Lambda}}(f_1f_2)\cdots\phi_{\omega_{\Lambda}}(f_{k-1}f_0)\\
&=(\omega_{\Lambda}(f_0,f_0f_1)\cdot\mathfrak{s}\cdot\omega_{\Lambda}(f_1,f_0f_1)^{-1})(\omega_{\Lambda}(f_1,f_1f_2)\cdot\mathfrak{s}\cdot\omega_{\Lambda}(f_2,f_1f_2)^{-1})\cdots\\
& \qquad\qquad(\omega_{\Lambda}(f_{k-1},f_{k-1}f_0)\cdot\mathfrak{s}\cdot\omega_{\Lambda}(f_0,f_{k-1}f_0)^{-1})\\
&=\omega(w_1,f_0)^{-1}\cdot\mathfrak{s}\cdot\omega(w_1,f_1)\cdot\omega(w_2,f_1)^{-1}\cdot\mathfrak{s}\cdot\omega(w_2,f_2)\cdots\\
& \qquad\qquad\omega(w_k,f_{k-1})^{-1}\cdot\mathfrak{s}\cdot\omega(w_k,f_0)\\
&=\omega(w_1,f_0)^{-1}\cdot\mathfrak{s}\cdot\omega(w_1,f_1)\cdot\mathfrak{s}\cdot\omega(w_2,f_1)^{-1}\omega(w_2,f_2)\cdot\mathfrak{s}\cdots\\
& \qquad\qquad\mathfrak{s}\cdot\omega(w_k,f_{k-1})^{-1}\omega(w_k,f_0)\\
&=\omega(w_1,f_0)^{-1}\cdot\mathfrak{s}\cdot\phi_1(f_1)\phi_1(f_2)\cdots\phi_1(f_{k-1})\cdot\omega(w_k,f_0)\\
&=\phi_1(f_1)\phi_1(f_2)\cdots\phi_1(f_{k-1})\omega(w_k,f_0)\cdot\mathfrak{s}\cdot\omega(w_1,f_0)^{-1}\\
&=\phi_1(f_1)\phi_1(f_2)\cdots\phi_1(f_{k-1})\phi_1(f_0)\\
&=\phi_1(W_{\Gamma}).
\end{align*}
Thus, $\phi_{\omega_{\Lambda}}(\vec{C}_{e_0})=\mathfrak{s}^{l-k}\phi_1(W_{\Gamma})$.  Similarly, we can calculate $\phi_{\kappa_{\Lambda}}(\vec{C}_{e_0})=\mathfrak{s}^{l-k}\phi_2(W_{\Gamma})$.

Since $\mathfrak{G}^{\mathfrak{s}}$ is abelian, Proposition \ref{conjugateWalk} says that switching a $\mathfrak{G}^{\mathfrak{s}}$-gain graph does not change the gain of a closed walk.  Hence, our assumption of $\Phi_1\sim \Phi_2$ implies $\phi_1(W_{\Gamma})=\phi_2(W_{\Gamma})$.  Therefore, $\phi_{\omega_{\Lambda}}(\vec{C}_{e_0})=\mathfrak{s}^{l-k}\phi_1(W_{\Gamma}) = \mathfrak{s}^{l-k}\phi_2(W_{\Gamma})= \phi_{\kappa_{\Lambda}}(\vec{C}_{e_0})$.  Since $C$ was arbitrary, the proof is complete by Lemma \ref{SwitchingFExistence}.
\end{proof}

An arbitrary $\mathfrak{G}^{\mathfrak{s}}$-gain graph $\Phi$ has many possible orientations $\omega$.  By selecting one particular orientation, we can then produce a single line graph as above.  However, the gains in the line graph depend on the original chosen orientation.  Therefore, the line graph of a $\mathfrak{G}^{\mathfrak{s}}$-gain graph cannot be a single gain graph.  Theorem \ref{ReOrSwEqLG} allows us to instead define the line graph of $\Phi$ as switching class.  If $\omega$ is an arbitrary orientation of $\Phi$, then the \emph{line graph of $\Phi$}, written $\Lambda(\Phi)$, is the switching class $[\Phi(\omega_{\Lambda})]$.  See Figure \ref{LinegraphExample} for an example.  We can even define the line graph of a switching class $[\Phi]$ as $[\Phi(\omega_{\Lambda})]$; that is, $\Lambda([\Phi])=[\Phi(\omega_{\Lambda})]$.  If $\mathfrak{G}=\{+1,-1\}$ with $\mathfrak{s}=-1$, then these definitions generalize line graph of a signed graph \cite{MananthavadyNotes}.



\begin{figure}[h!]
    \includegraphics[scale=0.6]{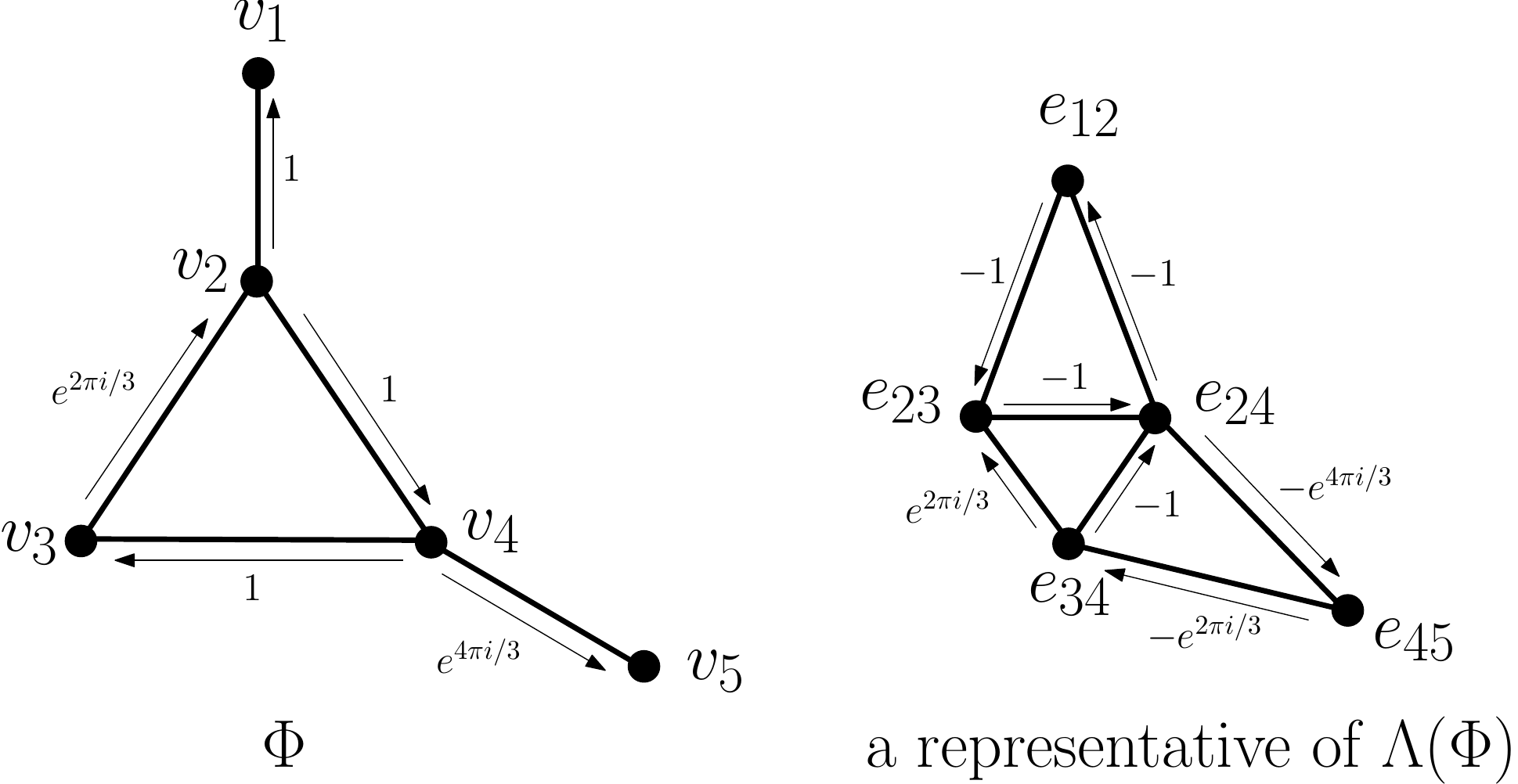}\centering
    \caption{A $\mathbb{T}^{-1}$-gain graph $\Phi$, and a representative of $\Lambda(\Phi)$.}\label{LinegraphExample}
\end{figure}

\noindent {\bf Question 2:} Is there an analogue of Theorem \ref{ReOrSwEqLG} for nonabelian $\mathfrak{G}^{\mathfrak{s}}$?
%
\section{Matrices}\label{SECTIONMatricesofORIENTEDGAINGRAPHS}

Now we study several matrices associated to the various gain graphic structures defined above.  Several of these matrices are new and generalize some previously known concepts for graphs, signed graphs and $\mathbb{T}$-gain graphs.

Let $\Phi$ be a $\mathfrak{G}^{\mathfrak{s}}$-gain graph.  The \emph{adjacency matrix} $A(\Phi)=(a_{ij})$ is an $n\times n$ matrix with entries in $\mathfrak{G}^{\mathfrak{s}}\cup\{0\}$, and is is defined by
\begin{equation*}
a_{ij}=
\begin{cases} \phi(e_{ij}) & \text{if }v_i\text{ is adjacent to }v_j,
\\
0 &\text{otherwise.}
\end{cases}
\end{equation*}

Let $\Phi$ be an $\mathfrak{G}^{\mathfrak{s}}$-gain graph.  The \emph{incidence matrix} $\Eta(\Phi)=(\eta_{ve})$ is an  $n\times m$ matrix with entries in $\mathfrak{G}^{\mathfrak{s}}\cup\{0\}$, defined by
\begin{equation*}
\eta_{v_i e}=
\begin{cases} \eta_{v_j e}\cdot\mathfrak{s}\cdot\phi(e_{ij}) &\text{if } e=e_{ij} \in E,
\\
0 &\text{otherwise;}
\end{cases}
\end{equation*}
furthermore, $\eta_{v_i e}\in \mathfrak{G}^{\mathfrak{s}}$ if $e_{ij}\in E$.  We say ``an" incidence matrix, because with this definition $\Eta(\Phi)$ is not unique.  Each column can be left multiplied by any element in $\mathfrak{G}^{\mathfrak{s}}$ and the result can still be called an incidence matrix.  For example, we can choose $\eta_{v_je}=1_{\mathfrak{G^{\mathfrak{s}}}}$ so $\eta_{v_i e}=\mathfrak{s}\cdot\phi(e_{ij})$ for each $e=e_{ij}\in E$.

Providing a $\mathfrak{G}^{\mathfrak{s}}$-gain graph $\Phi$ with an orientation results in a well defined incidence matrix for the oriented $\mathfrak{G}^{\mathfrak{s}}$-gain graph $(\Phi,\omega)$.  The \emph{incidence matrix} $\Eta(\Phi,\omega)=(\eta_{ve})$ is an $n\times m$ matrix with entries in $\mathfrak{G}^{\mathfrak{s}}\cup\{0\}$, defined by
\begin{equation}\label{OCUGGincidencerel}
\eta_{v e}=\omega(v,e) \text{ for every } (v,e) \in V\times E.
\end{equation}
The adjacency matrix of $\Lambda(\Phi)$ is not well defined since the gain of a specific oriented edge is unknown.  However, the adjacency matrix of $(\Phi(\omega_{\Lambda}),\omega_{\Lambda})$ is well defined because the gain of every oriented edge is determined by $\omega_{\Lambda}$. 

\subsection{Complex Unit Gain Graphs}
We can now state the following generalization of a signed graphic result due to Zaslavsky \cite{MITTOSSG} to the setting of complex unit gain graphs.  This also generalizes the well known relationship between the oriented incidence matrix of a graph and the adjacency matrix of the corresponding line graph. 

\begin{thm}\label{LineGraphHAEq} Let $(\Phi,\omega)$ be an oriented $\mathbb{T}^{\mathfrak{s}}$-gain graph, where $\mathfrak{s}$ is fixed as either $+1$ or $-1$.  Then
\begin{equation}
\Eta(\Phi,\omega)^*\Eta(\Phi,\omega)=2I+\mathfrak{s}A(\Phi(\omega_{\Lambda}),\omega_{\Lambda}).
\end{equation}
\end{thm}
\begin{proof}
Notice that $\Eta(\Phi,\omega)^*\Eta(\Phi,\omega)$ is an $E\times E$ matrix.  Consider the dot product of row $\mathbf{r}_i$ of $\Eta(\Phi,\omega)^*$ with column $\mathbf{c}_j$ of $\Eta(\Phi,\omega)$. Suppose column $\mathbf{c}_i$ corresponds to edge $e_{qr}$ and $\mathbf{c}_j$ corresponds to edge $e_{lk}$.  Figure \ref{OLinegraphExample} is a helpful reference for the following calculation.

Case 1: $i=j$ (same edge).  Then $\mathbf{r}_i=\mathbf{c}_j^*$ and thus,
\begin{align*}
\mathbf{r}_i\cdot \mathbf{c}_j=\mathbf{c}_j^*\cdot \mathbf{c}_j&=\bar{\eta}_{v_ke_{lk}}\eta_{v_ke_{lk}}+ \bar{\eta}_{v_le_{lk}}\eta_{v_le_{lk}}\\ 
&=\overline{\omega(v_k,e_{lk})}\omega(v_k,e_{lk})+ \overline{\omega(v_l,e_{lk})}\omega(v_l,e_{lk})\\
&=|\omega(v_k,e_{lk})|^2+ |\omega(v_l,e_{lk})|^2\\
&=2.
\end{align*}

Case 2: $i\neq j$ and $r=l$ (distinct adjacent edges).
\begin{align*}
\mathbf{r}_i\cdot \mathbf{c}_j=\mathbf{c}_i^*\cdot \mathbf{c}_j&= \bar{\eta}_{v_l e_{ql}}\eta_{v_le_{lk}}\\
&= \omega(v_l,e_{ql})^{-1} \omega(v_l,e_{lk})\\ 
&= \omega_{\Lambda}(e_{ql},e_{ql}e_{lk})\omega_{\Lambda}(e_{lk},e_{ql}e_{lk})^{-1}\\
&=\mathfrak{s}\phi_{\Lambda}(e_{ql}e_{lk}).
\end{align*}  
Therefore, $\Eta(\Phi,\omega)^*\Eta(\Phi,\omega)=2I+\mathfrak{s}A(\Phi(\omega_{\Lambda}),\omega_{\Lambda})$.
\end{proof}

\begin{lem}[\cite{MR2900705}, Lemma 4.1]\label{simASwitch} Let $\Phi_1=(\Gamma,\phi_1)$ and $\Phi_2=(\Gamma,\phi_2)$ both be $\mathbb{T}$-gain graphs. If $\Phi_1\sim \Phi_2$, then $A(\Phi_1)$ and $A(\Phi_2)$ have the same spectrum.
\end{lem}

Even though $A(\Lambda(\Phi))$ is only well defined up to switching, its eigenvalues are well defined by Lemma \ref{simASwitch}.

\begin{cor}\label{LGAdjSpec}
Let $(\Phi,\omega)$ be an oriented $\mathbb{T}^{\mathfrak{s}}$-gain graph, where $\mathfrak{s}$ is fixed as either $+1$ or $-1$. Then $A(\Lambda(\Phi))$ and $A(\Phi(\omega_{\Lambda}),\omega_{\Lambda})$ have the same spectrum.
\end{cor}


For a graph, there is a classic bound on the eigenvalues which says that all adjacency eigenvalues of the line graph are greater than or equal to $-2$.  For a signed graph, a similar statement can be made, which says that all adjacency eigenvalues of the line graph of a signed graph are less than or equal to 2 \cite{MITTOSSG}.  A generalization for $\mathbb{T}^{-1}$-gain graphs can be stated, establishing an upper bound on the eigenvalues of $A(\Lambda(\Phi))$.
\begin{thm}\label{LHEignM1} Let $\Phi=(\Gamma,\phi)$ be a $\mathbb{T}^{-1}$-gain graph.  If $\lambda$ is an eigenvalue of $A(\Lambda(\Phi))$, then $\lambda \leq 2$.
\end{thm}
\begin{proof}
Let $\omega$ be an arbitrary orientation of $\Phi$.  Suppose that $\mathbf{x}$ is an eigenvector of $A(\Phi(\omega_{\Lambda}),\omega_{\Lambda})$ with associated eigenvalue $\lambda$.  By Theorem \ref{LineGraphHAEq} the following simplification can be made:
\[ \Eta(\Phi,\omega)^*\Eta(\Phi,\omega) \mathbf{x} =\big (2I-A(\Phi(\omega_{\Lambda}),\omega_{\Lambda})\big)\mathbf{x} = (2-\lambda)\mathbf{x}.\]
Therefore, $2-\lambda$ is an eigenvalue of $\Eta(\Phi,\omega)^*\Eta(\Phi,\omega)$.  Since $\Eta(\Phi,\omega)^*\Eta(\Phi,\omega)$ is positive semidefinite it must be that $2-\lambda\geq 0$ and therefore, $2\geq \lambda$.  By Corollary \ref{LGAdjSpec}, the result follows.
\end{proof}

Similarly for $\mathbb{T}^{1}$-gain graphs a lower bound on the eigenvalues of $A(\Lambda(\Phi))$ may be achieved.
\begin{thm}\label{LHEignP1} Let $\Phi=(\Gamma,\phi)$ be a $\mathbb{T}^{1}$-gain graph.  If $\lambda$ is an eigenvalue of $A(\Lambda(\Phi))$, then $-2 \leq \lambda$.
\end{thm}

\noindent {\bf Question 3:} Can we classify all $\mathbb{T}^{-1}$-gain graphs with adjacency eigenvalues less than 2?  Similarly, can we classify all $\mathbb{T}^{1}$-gain graphs with adjacency eigenvalues grater than $-2$?  Perhaps one might consider these questions for the group of $n$th roots of unity $\boldsymbol\mu_n$ instead of the full circle group $\mathbb{T}$.  The bound in Theorem \ref{LHEignM1} holds for $\boldsymbol\mu_{2n}^{-1}$-gain graphs and the bound in Theorem $\ref{LHEignP1}$ holds for all $\boldsymbol\mu_{n}^{1}$-gain graphs.

  Recently, Krishnasamy, Lehrer
and Taylor have classified indecomposable star-closed line systems in $\mathbb{C}^n$ \cite{MR2470539, MR2542964}, which vastly generalizes the results of Cameron, Goethals, Seidel, and Shult \cite{MR0441787} as well as Cvetkovi{\'c}, Rowlinson, and Simi{\'c} \cite{MR2120511}.  If there is an answer to the above in connection to these indecomposable line systems in $\mathbb{C}^n$, a study of the exceptional graphs could generalize the related work of Chawathe and G.R. Vijayakumar on signed graphs \cite{MR1164766, MR1234728, MR884068, MR1078708}.


\section{ Acknowledgements}
The author would like to thank Thomas Zaslavsky and Marcin Mazur for their valuable comments and suggestions regarding this work.

\bibliographystyle{amsplain2}
\bibliography{mybib}

\end{document}